\documentclass[10pt]{amsart}
\usepackage[margin=1in,letterpaper,portrait]{geometry}
\usepackage{xypic,amsmath,amssymb,verbatim,hyperref}
\usepackage[curve]{xy}
\xyoption{frame}

\newtheorem{theorem}{Theorem}
\numberwithin{theorem}{section}
\newtheorem{lemma}[theorem]{Lemma}
\newtheorem{proposition}[theorem]{Proposition}
\newtheorem{corollary}[theorem]{Corollary}

\theoremstyle{definition}
\newtheorem{definition}[theorem]{Definition}
\newtheorem{example}[theorem]{Example}
\newtheorem{remark}[theorem]{Remark}
\newtheorem{question}[theorem]{Question}

\DeclareMathOperator{\coker}{coker}
\DeclareMathOperator{\im}{im}
\DeclareMathOperator{\diag}{diag}
\DeclareMathOperator{\pdet}{pdet}

\DeclareMathOperator{\rank}{rank}
\DeclareMathOperator{\op}{{op}}
\DeclareMathOperator{\sgn}{sgn}
\DeclareMathOperator{\Sd}{Sd}
\DeclareMathOperator{\Bd}{Bd}
\DeclareMathOperator{\Trace}{trace}
\DeclareMathOperator{\one}{\mathbf{1}}
\DeclareMathOperator{\RowB}{RowB}
\DeclareMathOperator{\ColB}{ColB}

\newcommand{\isom}{\cong}
\newcommand{\excise}[1]{}
\newcommand{\half}{\frac{1}{2}}
\newcommand{\tr}{t}
\newcommand{\opendualblock}{D^\circ}
\newcommand{\dualblock}{D}

\newcommand{\Ss}{\mathbb{S}}
\newcommand{\Zz}{\mathbb{Z}}
\newcommand{\HH}{\tilde H}

\newcommand{\x}{\times}
\newcommand{\xx}{\mathbf{x}}
\newcommand{\yy}{\mathbf{y}}
\newcommand{\zz}{\mathbf{z}}

\newcommand{\bd}{\partial}

\newcommand{\sm}{\setminus}

\newcommand{\ZZ}{\mathbb{Z}}
\newcommand{\QQ}{\mathbb{Q}}
\newcommand{\RR}{\mathbb{R}}

\newcommand{\PPP}{\mathcal{P}}

\newcommand{\Hom}{\operatorname{Hom}}
\newcommand{\Ext}{\operatorname{Ext}}

\author{Jeremy L.\ Martin}
\address{Department of Mathematics, University of Kansas, Lawrence, KS 66045}
\email{jlmartin@ku.edu}
\author{Molly Maxwell}
\address{Flathead Valley Community College, Kalispell, MT 59901}
\email{mmaxwell@fvcc.edu}
\author{Victor Reiner}
\address{School of Mathematics, University of Minnesota, Minneapolis MN 55455}
\email{reiner@math.umn.edu}
\author{Scott O.\ Wilson}
\address{Department of Mathematics, Queens College, CUNY, Queens, NY 11367}
\email{Scott.Wilson@qc.cuny.edu}
\thanks{J.L.\ Martin was supported in part by a Simons Foundation Collaboration Grant and by National Security Agency grant no. H98230-12-1-0274.
V. Reiner was supported by NSF grant DMS-1001933.}
\date{\today}
\title{Pseudodeterminants and perfect square spanning tree counts}
\begin{document}
\maketitle

\begin{abstract}
The pseudodeterminant \(\textrm{pdet}(M)\) of a square matrix is the last nonzero coefficient in its characteristic polynomial; for a nonsingular matrix, this is just the determinant.  If \(\partial\) is a symmetric or skew-symmetric matrix then \(\textrm{pdet}(\partial\partial^t)=\textrm{pdet}(\partial)^2\).  Whenever \(\partial\) is the \(k^{th}\) boundary map of a self-dual CW-complex \(X\), this linear-algebraic identity implies that the torsion-weighted generating function for cellular \(k\)-trees in \(X\) is a perfect square.  In the case that \(X\) is an \emph{antipodally} self-dual CW-sphere of odd dimension, the pseudodeterminant of its \(k\)th cellular boundary map can be interpreted directly as a torsion-weighted generating function both for \(k\)-trees and for \((k-1)\)-trees, complementing the analogous result for even-dimensional spheres given by the second author.  The argument relies on the topological fact that any self-dual even-dimensional CW-ball can be oriented so that its middle boundary map is skew-symmetric.
\end{abstract}
\section{Introduction}
\label{intro-section}

This paper is about generating functions for higher-dimensional spanning tree
via determinants of combinatorial Laplacians, and when they are
perfect squares.  Combinatorial Laplacians, broadly interpreted, are matrices 
of the form $L=\bd \bd^\tr$ where $\bd$ is a matrix in $\ZZ^{n \times m}$,
or perhaps even $R^{n \times m}$ where $R$ is an integral domain containing
$\ZZ$ along with indeterminates used as weights.
Often $L$ is singular, so that instead of considering
the determinant $\det(L)$, one first creates an invertible
{\it reduced Laplacian} by striking out some rows
and columns.

We will explore a useful alternative approach using the notion of 
\emph{pseudodeterminant} \cite{Knill, Minka}, appearing perhaps earliest
in work of Adin \cite[Theorem 3.4]{Adin}.
One first defines the rank $r$ of $L$ by extending scalars to the 
fraction field $K$ of the domain $R$.
Thus $L$ will have $r$ nonzero eigenvalues $\lambda_1,\ldots,\lambda_r$ 
in the algebraic closure $\overline{K}$ of
$K$, leading to two expansions for its
unsigned characteristic polynomial:
\begin{align}
\label{principal-minor-expansion}
\det(t\one + L) &= \sum_{I \subseteq \{1,2,\ldots,n\}} t^{n-|I|} \det(L_{I,I}) 
             =t^n + \Trace(L) t^{n-1} + \cdots + \det(L) t^0 \\
\label{eigenvalue-expansion}
&= \prod_{i=1}^r (t+\lambda_i)
\end{align}
where $L_{I,I}$ is the {\it principal square submatrix} of $L$ indexed by row and column
indices in the subset $I$.

\begin{definition}  
The \emph{pseudodeterminant} $\pdet(L)$ is the last nonzero coefficient in the unsigned characteristic polynomial of $L$.
That is,
\begin{equation}
\label{pdet-definition}
\pdet(L) := 
  \sum_{\substack{I \subseteq [n]: \\ |I| = r}} \det(L_{I,I})
= \prod_{i=1}^r \lambda_i.
\end{equation}
\end{definition}

\noindent
Thus for nonsingular $L$ one has $\pdet(L)=\det(L)$, and when $L$ is of rank one it has $\pdet(L)=\Trace(L)$.
The pseudodeterminant is also the leading coefficient in
the {\it Fredholm determinant} $\det(\one + tL)$ 
for the scaled operator $tL$, that is,
$\det(\one + tL)=\prod_{i=0}^n (1+ t\lambda_i)
= \prod_{i=0}^r (1+ t\lambda_i)= \pdet(L) t^r + O(t^{r-1})$;
see, e.g., Simon \cite[Chap. 3]{Simon}.

Section~\ref{char-poly-section} quickly reviews and reformulates
the well-known expansion of $\det(t \one + L)$ via the Binet-Cauchy theorem. 
Section~\ref{Dirac-section} digresses
to explain a general perfect square phenomenon occurring if $\bd^2=0$: 
\begin{equation}
\pdet(\bd \bd^\tr) = \pdet( \bd+\bd^\tr) ^2.
\end{equation}
This happens because $\bd + \bd^\tr$ plays the 
role of a combinatorial {\it Dirac operator},
whose square gives a symmetrized Hodge-theoretic version of
the combinatorial Laplacian:
\begin{equation}
\label{Dirac-squared-is-L}
(\bd+\bd^\tr)^2 = (\bd^2 + \bd\bd^\tr + \bd^\tr\bd + (\bd^2)^\tr)
                        = \bd\bd^\tr + \bd^\tr\bd.
\end{equation}

Section~\ref{symmetry-section} proves a simple linear algebra result,
Theorem~\ref{perfect-square-theorem},
about the situation whenever $\bd$ in $\ZZ^{n \times n}$
is symmetric or skew-symmetric, that is, $\bd^\tr=\pm \bd$: one then has
\begin{equation} \label{pdet-squared}
\pdet(\bd \bd^\tr) = (\pdet\bd)^2.
\end{equation}
When $\bd$ is the $i^{th}$ boundary map of a CW-complex~$X$, the summands in \eqref{pdet-definition} can be interpreted in terms of \emph{cellular spanning trees}. This theory stems from the work of Kalai \cite{Kalai} and Bolker \cite{Bolker}
and has been developed in recent work such as \cite{Adin, BajoBurdickChmutov, DKM-Simplicial,  DKM-Cellular,  DKM-Cutflow, KrushkalRenardy, Lyons, Maxwell}; we review it briefly in Section~\ref{trees-section}.  In the special case that $\bd$ is symmetric or skew-symmetric, the linear-algebraic identity \eqref{pdet-squared} says that the torsion-weighted number $\tau_i(X)$ of \emph{cellular spanning trees} of~$X$,
\[\tau_i(X) :=\sum_{i\text{-trees }T\subseteq X} |\HH_{i-1}(T)|^2,\]
is a perfect square, as is a more general generating function for $i$- and $(i-1)$-trees.
  
The second author \cite{Maxwell} explicated results of Tutte \cite{Tutte} and a question of Kalai \cite[\S 7, Prob.~3]{Kalai}, by showing
that certain even-dimensional antipodally self-dual CW-spheres have spanning tree
counts that are perfect squares, with a combinatorially significant square root.  A goal of this paper is to prove similar results for odd-dimensional CW-spheres.

In Section~\ref{duality-section}, we consider cellular $d$-balls $S$ whose face posets are self-dual (a relatively weak self-duality condition), with no constraint on the parity of their dimension.  Using Alexander duality, we show (Proposition~\ref{Alexander-dual-higher-dimensional-trees}) that $\tau_i(S)=\tau_{d-1-i}(S)$, as well as an analogous formula for weighted tree counts.  These results were observed by Kalai \cite{Kalai} for simplices, and our proof in the general case is based on Kalai's ideas.  A consequence of these results is that the pseudodeterminant of the (weighted or unweighted) middle Laplacian of
a self-dual complex is always a perfect square (Corollary~\ref{perfect-square-corollary}).

In Section~\ref{antipodally-self-dual-section}, we study the specific case of antipodally self-dual odd-dimensional spheres $S\isom\Ss^{2k-1}$.  Such a sphere can always be oriented so that the middle boundary map $\bd=\bd_k$ satisfies $\bd^\tr=(-1)^k \bd$.  Together with Theorem~\ref{perfect-square-theorem}, this implies that $\pdet(\bd)$ has a direct combinatorial interpretation as $\tau_k(S) = \tau_{k-1}(S)$; the weighted analogue of this statement is also valid (Theorem~\ref{antipodally-self-dual-square-root-theorem}).  The construction of the required orientation is technical and is deferred to an appendix (Section~\ref{antipodally-self-dual-appendix}), although it can be made combinatorially explicit for certain spheres including polygons and boundaries of simplices.

\section{Unsigned characteristic polynomials of Laplacians}
\label{char-poly-section}

Let $R$ be an integral domain, and let $\bd\in R^{n\x m}$ (that is, $\bd$ is an $n\x m$ matrix over $R$).  Let $I \subseteq [n]:=\{1,2,\ldots,n\}$ be a set of row indices and let
$J \subseteq [m]:=\{1,2,\ldots,m\}$ be a set of column indices.
The pair $I,J$ determines an submatrix $\bd_{I,J}$ in $R^{|I| \times |J|}$.
Label row and column indices of $\bd$ in $R^{n \times m}$ with 
indeterminates $\xx:=(x_1,\ldots,x_n)$ and $\yy:=(y_1,\ldots,y_m)$.
Let $X=\diag(\xx)$ be the square diagonal matrix having $\xx$ as its
diagonal entries, and likewise let $Y=\diag(\yy)$.
For subsets
$I \subseteq [n]$ and $J \subseteq [m]$, 
define monomials
$\xx^I:=\prod_{i \in I} x_i$
and
$\yy^J:=\prod_{j \in J} y_j.$

The following elementary proposition is well known.  However, we were unable to find an explicit reference in the literature for its weighted version, so we include a proof for the sake of completeness.

\begin{proposition} 
\label{weighted-Kelmans-expansion}
Every matrix $\bd\in R^{n \times m}$ satisfies
\begin{equation}
\label{Kelmans-expansion}
\det(t\one + \bd \bd^\tr) 
  = \sum_{ \substack{ I\subseteq[n],\;J\subseteq[m]:\\ |I|=|J|} } 
         t^{n-|I|} \left( \det \bd_{I,J} \right)^2
\end{equation}
and more generally, the weighted operator $L:=X^{\half} \bd Y \bd^\tr X^{\half}\in (R[\xx,\yy])^{n\x n}$ satisfies
\begin{equation} \label{weighted-expansion}
\det(t\one + L ) 
  = \sum_{ \substack{ I\subseteq[n],\;J\subseteq[m]:\\ |I|=|J|} } 
           t^{n-|I|} \xx^I \yy^J \left( \det \bd_{I,J} \right)^2.
\end{equation}
In particular, if $\bd$ has rank $r$, then
\begin{equation}
\label{weighted-pdet-expansion}
\pdet(L) =
\sum_{ \substack{ I\subseteq[n],\;J\subseteq[m]:\\ |I|=|J|=r} } 
         \xx^I \yy^J \left( \det \bd_{I,J} \right)^2.
\end{equation}
\end{proposition}

\begin{proof}
Let
$Z:=X^{\half} \bd Y^{\half}$,
so that
$L:=X^{\half} \bd Y \bd^\tr X^{\half}=Z Z^\tr$.
The Binet-Cauchy identity gives 
$$
\det L_{I,I}
= \sum_{J \subseteq [m]:\ |J|=|I|}
            \det Z_{I,J} \det Z^\tr_{J,I}
$$
and the principal minor expansion 
\eqref{principal-minor-expansion} gives
\begin{align*}
\det(t \one + L )
  &~=~ \sum_{I \subseteq [n]} t^{n-|I|} \det L_{I,I} 
   ~=~ \sum_{I \subseteq [n]}  t^{n-|I|}
     \sum_{J \subseteq [m]:\ |J|=|I|} 
           \det Z_{I,J} \det Z^\tr_{J,I}\\
  &~=~ \sum_{\substack{I\subseteq[n],\;J\subseteq [m]:\\ |J|=|I|}} t^{n-|I|}
       \left( (\xx^I)^{\half} \cdot \det\bd_{I,J} \cdot (\yy^J)^{\half} \right)  
       \left( (\yy^I)^{\half} \cdot \det\bd^\tr_{I,J} \cdot (\xx^J)^{\half} \right) \\  
  &~=~ \sum_{\substack{I\subseteq[n],\;J\subseteq [m]:\\ |J|=|I|}} 
          t^{n-|I|} \xx^I \yy^J \left( \det \bd_{I,J} \right)^2,
\end{align*}
proving~\eqref{weighted-expansion}. Setting $x_i=y_j=1$ for all $i,j$ recovers~\eqref{Kelmans-expansion}.
To obtain~\eqref{weighted-pdet-expansion}, note that $R[\xx,\yy]$
is an integral domain, and $\bd, \bd \bd^\tr, L$ all have rank~$r$.
\end{proof}

\noindent
The nonzero summands in \eqref{weighted-pdet-expansion}
are those for which $\bd_{I,J}$ is nonsingular.
Accordingly, we can reformulate the
summation indices in \eqref{weighted-pdet-expansion}.

\begin{definition}
\label{row-column-basis-definition}
For a domain $R$ and 
$\bd$ in $R^{n \times m}$, say that a subset of row indices $I \subseteq [n]$ 
forms a {\it row basis for $\bd$} if, after extending
scalars to the fraction field $K$ of $R$, the rows of $\bd$
indexed by $I$ give a $K$-vector space basis for the row space of $\bd$.
Similarly define for $J \subseteq [m]$
what it means to be a {\it column basis for $\bd$}.
We will write $\RowB(\bd)$ and $\ColB(\bd)$ for the set of row and column bases,
respectively.
\end{definition}

\begin{proposition}
\label{row-basis-column-basis-prop}
(cf. \cite[Chap. 4, Exer. 2.5]{Artin})
Let $R$ be a domain and $\bd$ in $R^{n \times m}$ of rank $r$,
and let $I\subseteq[n]$, $J\subseteq[m]$ have $|I|=|J|=r$.
Then the submatrix $\bd_{I,J}$ is nonsingular if and only if 
both $I$ is a row basis and $J$ is a column basis for $\bd$.
\end{proposition}


\begin{proof}
Extending scalars from $R$ to $K$, factor the $K$-linear map
$\bd_{I,J}:K^J \rightarrow K^I$ into 
$\bd_{I,J}=\beta\circ\alpha$ as follows:
$$
\xymatrix{
K^J \ar@{=}[dd] \ar@{^{(}->}[r] \ar@(ul,ur)[rrrr]^{\bd_{I,J}} & K^m \ar@{>>}[dr] \ar[rr]^{\bd} & & K^n \ar@{>>}[r]^{\pi} & K^I  \ar@{=}[dd]\\
 & & \im\bd \ar@{=}[d] \ar@{^{(}->}[ur]^{i} & & \\
K^r \ar^{\alpha}[rr]    &                 & K^r \ar^{\beta}[rr] &            & K^r \\
}
$$
In the top row, the first horizontal inclusion $K^J \hookrightarrow K^m$ pads a vector in $K^J$ with extra zero coordinates outside of $J$ to create 
a vector in $K^m$, while the last horizontal surjection 
$K^n \twoheadrightarrow K^I$ forgets the coordinates outside
of $I$.  The factorization $\bd_{I,J}=\beta\circ\alpha$ 
shows that $\bd_{I,J}$ is nonsingular if and only if 
both $\alpha$ and $\beta$ are nonsingular, that is, if and only if both
$J$ is a column basis and $I$ is a row basis for $\bd$.
\end{proof}

Proposition~\eqref{row-basis-column-basis-prop} allows us to rewrite equation~\eqref{weighted-pdet-expansion} as follows:
\begin{equation}
\label{reformulated-weighted-pdet-expansion}
\pdet(L) 
  = \sum_{\substack{I\in\RowB(\bd)\\ J\in\ColB(\bd)}}
         \xx^I \yy^J \left( \det \bd_{I,J} \right)^2.
\end{equation}

\section{Digression:  Pseudodeterminants and Laplacians as squares of Dirac operators}
\label{Dirac-section}

This section will not be used in the sequel.
We first collect a few easy properties of pseudodeterminants analogous to 
properties of determinants,  then apply them 
to show why the pseudodeterminant of the
{\it Dirac operator} $\bd + \bd^\tr$, defined for 
$\bd$ in $\ZZ^{n \times n}$ satisfying $\bd^2=0$,
agrees up to $\pm$ sign with 
the pseudodeterminant for either of the Laplace operators
$\bd \bd^\tr$ or $\bd^\tr \bd$.

As usual, $R$ will be a domain, with fraction field $K$ having 
algebraic closure $\overline{K}$.

\begin{proposition}(cf.\ Knill \cite[Prop.~2]{Knill}) 
\label{pdet-properties}
For $L$ in $R^{n \times n}$, one has the following.
\begin{enumerate}
\item[(a)]
$\pdet(L^\tr)=\pdet(L)$.
\item[(b)]
$\pdet(L^k)=\pdet(L)^k$ for $k=1,2,\ldots$.
\item[(c)]
If $A,B$ lie in $R^{n \times m}, R^{m \times n}$, respectively,
then $\pdet(AB)=\pdet(BA)$.
\item[(d)]
If $L,M$ in $R^{n \times n}$ are mutually annihilating (i.e., $LM=0=ML$), then 
$$
\pdet(L+M) = \pdet(L) \pdet(M).
$$
\end{enumerate}
\end{proposition}
\begin{proof}
Assertion (a) follows from 
$\det(t \one + L^\tr ) = \det\left((t \one + L)^\tr \right)=\det(t \one + L)$.

Assertion (b) follows since if $L$ has nonzero eigenvalues $\lambda_1,\ldots,\lambda_r$, then
$L^k$ has nonzero eigenvalues $\lambda_1^k,\ldots,\lambda_r^k$.

Assertion (c) comes from a well-known determinant fact
(see, e.g., \cite{Schmid}) asserting that, if $n \geq m$, then 
$\det(t \one + AB) =t^{n-m} \det(t \one + BA)$.

Assertion (d) will follow by making a change of coordinates
in $\overline{K}^n$ that simultaneously triangularizes the mutually
annihilating (and hence mutually commuting) matrices $L,M$.
Thus without loss of generality, $L$ and $M$ are triangular and
have as their ordered lists of diagonal entries their eigenvalues
$(\lambda_1,\ldots,\lambda_n), (\mu_1,\ldots,\mu_n)$.  
Their sum $L+M$ is then triangular, with
eigenvalues $(\lambda_1+\mu_1,\ldots,\lambda_n+\mu_n)$.
However, since either product
$ML$ or $LM$ is also triangular, with eigenvalues $(\lambda_1 \mu_1,\ldots,\lambda_n \mu_n)$, the
mutual annihilation $0=LM=ML$ implies that at most one of each 
pair $\{\lambda_i,\mu_i\}$ can be nonzero.  Hence if $L,M$ have ranks $r,s$, respectively, then the eigenvalues for $L+M$ can be reindexed as
$(\lambda_1,\ldots,\lambda_r,\mu_{r+1},\ldots,\mu_{r+s},0,0,\ldots,0)$
and hence
$$
\pdet(L+M)=\lambda_1 \cdots \lambda_r \cdot \mu_{r+1}\cdots \mu_{r+s}
=\pdet(L) \pdet(M). \qedhere
$$
\end{proof}

\begin{definition}
For $\bd$ in $\ZZ^{n \times n}$ with $\bd^2=0$,
its {\it Dirac operator} is the symmetric matrix
$\bd + \bd^\tr$.
\end{definition}
\noindent
The reader is referred to Friedrich \cite{Friedrich} for background on
Dirac operators in Riemannian geometry.  

As noted in equation \eqref{Dirac-squared-is-L} in the Introduction,
this operator $\bd + \bd^\tr$ has square given by
$$
\Delta := (\bd+\bd^\tr)^2 = (\bd^2 + \bd\bd^\tr + \bd^\tr\bd + (\bd^2)^\tr)
                        = \bd\bd^\tr + \bd^\tr\bd
$$
which is another form of combinatorial Laplacian, arising in 
discrete {\it Hodge theory} over $\RR$; see, e.g.,
Friedman \cite{Friedman}. The subspace
of {\it harmonics} $H:=\ker \Delta \subseteq \RR^n$ gives a canonical choice of representatives
for the homology $\ker\bd /\im \bd$, due to the orthogonal {\it Hodge decomposition} picture:
$$
\begin{matrix}
\RR^n                              &= &\im\bd^\tr & \oplus & H & \oplus &\im \bd\\
\ker\bd=\ker\bd^\tr\bd                     &= &              &        & H & \oplus &\im \bd\\
\ker\bd^\tr=\ker\bd\bd^\tr                   &= &\im\bd^\tr & \oplus & H &        &            \\
\ker\bd \cap \ker\bd^\tr &= &              &        & H.&        &            \\
\end{matrix}
$$

\begin{corollary}
\label{Dirac-interpretation}
Let  $\bd$ in $\ZZ^{n \times n}$ be such that 
$\bd^2=0$.  Then its Dirac operator $\bd+\bd^\tr$ satisfies
\[\pdet(\bd+\bd^\tr) ~=~\pm \pdet(\bd \bd^\tr) 
~=~
  \left( \pm \sum_{\substack{I\in\RowB(\bd)\\J\in\ColB(\bd)}}
         \left( \det \bd_{I,J} \right)^2 \right).
\]
\end{corollary}
\begin{proof}
The second equality comes from specializing all variables to $1$ in the right-hand side of~\eqref{reformulated-weighted-pdet-expansion}.
For the first equality, we check that its two sides
have the same square:
\begin{align*}
\left( \pdet (\bd + \bd^\tr) \right)^2 
  &= \pdet\left( (\bd + \bd^\tr)^2 \right)
   && \text{(Prop.~\ref{pdet-properties}(b))}\\
  &= \pdet\left( \bd\bd^\tr + \bd^\tr \bd \right)
   && \text{(Eqn.~\eqref{Dirac-squared-is-L})}\\
  &= (\pdet \bd\bd^\tr) (\pdet \bd^\tr\bd)
   && \text{(Prop.~\ref{pdet-properties}(d))}\\
  &= \pdet(\bd\bd^\tr)^2.
   && \text{(Prop.~\ref{pdet-properties}(c))}
\end{align*}
For the third equality, note that 
$\bd\bd^\tr$ and
$\bd^\tr\bd$ are mutually annihilating:
$(\bd^\tr\bd)(\bd\bd^\tr)=\bd^\tr(\bd^2)\bd^\tr=0$
and
$(\bd\bd^\tr)(\bd^\tr\bd)=\bd(\bd^2)^\tr\bd=0$.
\end{proof}

\begin{remark}
All four operators 
$\bd\bd^\tr$,
$\bd^\tr\bd$,
$\bd\bd^\tr + \bd^\tr\bd$, and
$\bd+\bd^\tr$ are self-adjoint.  The first three are positive semidefinite, and 
hence have nonnegative pseudodeterminant by  Proposition~\ref{pdet-properties}(c).
However, the Dirac operator $\bd+\bd^\tr$ can be indefinite and have
negative pseudodeterminant. For example, 
$
\bd=\left[ \begin{matrix} 0 & 1 \\ 0 & 0 \end{matrix} \right]
$
has $\bd^2=0$, and its Dirac operator
$
\bd+\bd^\tr = \left[ \begin{matrix} 0 & 1 \\ 1 & 0 \end{matrix} \right] 
$
has eigenvalues $(+1,-1)$, with $\pdet(\bd+\bd^\tr)=-1$.
\end{remark}

\section{Symmetry or skew-symmetry} 
\label{symmetry-section}

Something interesting happens to the pseudodeterminant
in our previous results when $\bd$ happens to be square and either symmetric or
skew-symmetric, due to the following fact.

\begin{lemma} 
\label{det-switch-lemma}
For $R$ a domain, with $\bd$ in $R^{n \times m}$ of rank $r$,
and any $r$-subsets $A,A' \subseteq [n]$ and
$B,B' \subseteq [m]$,
$$
\det \bd_{A,B}\det \bd_{A',B'}=\det \bd_{A',B}\det \bd_{A,B'}.
$$
In particular, when $n=m$, for any $r$-subsets $I,J$ of $[n]$ one has
$$
\det \bd_{I,I}\det \bd_{J,J}=\det \bd_{I,J}\det \bd_{J,I}.
$$
\end{lemma}

\begin{proof}
Extend scalars to the fraction field $K$ of $R$.
Considering $\bd$ as a $K$-linear map
$K^m \longrightarrow K^n$, its $r^{th}$ {\it exterior power} is
a $K$-linear map  
$$
\wedge^r K^m ~\xrightarrow{\wedge^r \bd}~ \wedge^r K^n.
$$
If $K^m, K^n$ have standard bases $(v_1,\ldots,v_m)$ and
$(w_1,\ldots,w_n)$, then $\wedge^r K^m,  \wedge^r K^n$ have
$K$-bases of wedges $v_A:=v_{a_1} \wedge \cdots \wedge v_{a_r}$
and $w_B:= w_{b_1} \wedge \cdots \wedge w_{b_r}$ indexed by $r$-subsets
$A \subseteq [m]$ and $B \subseteq [n]$.  The matrix for $\wedge^r \bd$
in these bases has $(A,B)$-entry $\det(\bd_{A,B})$.  To prove the lemma,
it suffices to show that 
$\wedge^r \bd$ has rank $1$,
so its $2 \times 2$ minors vanish.

Let us show more generally that $\wedge^k \bd$ has rank $\binom{r}{k}$.
Make changes of bases in $K^m, K^n$ via 
invertible matrices $P, Q$ in $GL_m(K), GL_n(K)$ ,
so that $\bd = P D Q$, where
$$
D=\left[
\begin{matrix}
I_r & 0 \\
0 & 0
\end{matrix} 
\right].
$$
Then 
$
\wedge^k \bd = \wedge^k PDQ = \wedge^k P \cdot \wedge^k D \cdot \wedge^k Q,
$
where $\wedge^k P, \wedge^k Q$ have inverses 
$\wedge^k(P^{-1}), \wedge^k(Q^{-1})$, and 
$$
\wedge^k D
=\left[
\begin{matrix}
I_{\binom{r}{k}} & 0 \\
0 & 0
\end{matrix} 
\right]
$$
clearly has rank $\binom{r}{k}$.  Hence this is also the rank of
$\wedge^k \bd$.
\end{proof}

Lemma~\ref{det-switch-lemma} has two interesting consequences for
matrices $\bd$ in $\ZZ^{n \times n}$ which are either symmetric or
skew-symmetric.  The first is the following observation about the pseudodeterminant
of such a matrix.

\begin{theorem}
\label{all-minors-same-sign}
If a matrix $\bd$ in $\ZZ^{n \times n}$ of rank $r$ has
$\bd^\tr=\pm \bd$, then all of its $r \times r$ 
principal minors $\{ \det(\bd_{I,I}) \}_{|I|=r}$
have the same sign, so that
\begin{align*}
\pdet(\bd) 
  ~&=~ \pm \sum_{I\in\RowB(\bd)} 
            |\coker(\bd_{I,I})|, \\
\pdet(X \bd) 
  ~&=~ \pm \sum_{I\in\RowB(\bd)} 
           \xx^I |\coker(\bd_{I,I})|.
\end{align*}
\end{theorem}
\begin{proof}
Combining the assumption $\bd^\tr=\pm \bd$ with
Lemma~\ref{det-switch-lemma} yields
\begin{align}
\det(\bd_{I,I})\det(\bd_{J,J})
 &=\det(\bd_{I,J}) \det(\bd_{J,I}) \notag\\
 &= \det(\bd_{I,J}) \det(\pm \bd_{I,J} ) \notag\\
 &= (\pm 1)^r \det(\bd_{I,J})^2 \notag\\
 &= \det(\bd_{I,J})^2 \label{det-switch-with-perfect-square}
\end{align}
where the last equality  comes from the fact that whenever $\bd^\tr=-\bd$,
the rank $r$ of $\bd$ must be even;
see Lang \cite[\S XIV.9]{Lang}.  Thus  the product of any two nonzero $r \times r$ principal minors
$\det(\bd_{I,I})$ and $\det(\bd_{J,J})$ is a perfect square, and in particular is positive,
so that their signs agree.  Therefore, all the summands in~\eqref{pdet-definition}
(replacing $L$ with $\bd$) have the same sign, and since $\det \bd_{I,I}=|\coker\bd_{I,I}|$, we obtain the desired formulas.
\end{proof}

The second consequence of Lemma~\ref{det-switch-lemma} is one of our main linear algebra results on perfect squares.

\begin{theorem}
\label{perfect-square-theorem}
Let $\bd$ in $\ZZ^{n \times n}$  such that
$\bd^\tr=\pm \bd$.  Then 
\begin{equation}
\label{unweighted-perfect-square-equation}
\pdet(\bd \bd^\tr) = \pdet(\bd)^2.
\end{equation}
More generally, the matrix $L:=X^{\half} \bd Y \bd^\tr X^{\half}$ satisfies
$$
\pdet(L) = \pdet(X \bd) \pdet(Y \bd^\tr).
$$
\end{theorem}
\begin{proof}
Combining equation~\eqref{det-switch-with-perfect-square}
with~\eqref{reformulated-weighted-pdet-expansion} gives
\begin{align*}
\pdet(L) 
  &= \sum_{  \substack{ I\in\RowB(\bd) \\ J\in\ColB(\bd) }}
         \xx^I \yy^J \det (\bd_{I,J})^2 \\
  &= \sum_{  \substack{ I\in\RowB(\bd) \\ J\in\ColB(\bd) }}
         \xx^I \yy^J \det(\bd_{I,I}) \det(\bd_{J,J}) \\
    &= \left( \sum_{I\in\RowB(\bd)} \xx^I \det(\bd_{I,I}) \right)
       \left( \sum_{J\in\ColB(\bd)} \yy^J \det(\bd_{J,J}) \right) 
       =\pdet(X \bd) \pdet(Y \bd^\tr). \qedhere
 \end{align*}
\end{proof}

\begin{remark}
When $r=n$, so that $\bd$ is nonsingular, we have $\pdet\bd=\det\bd$,
and Theorem~\ref{perfect-square-theorem} follows from the 
multiplicative property of $\det$, requiring no hypothesis that $\bd^\tr=\pm \bd$:
$$
\begin{aligned}
\det(L)&=\det(X^{\half} \bd Y \bd^\tr X^{\half})
= \det(X)^{\half} \det(\bd) \det(Y) \det(\bd^\tr) \det(X)^{\half} \\
&=\det(X) \det(\bd) \det(Y) \det(\bd^\tr) =\det(X \bd) \det( Y \bd^\tr).
\end{aligned}
$$
\end{remark}


\begin{remark}
\label{square-root-of-spectrum-remark}

It is easy to strengthen  \eqref{unweighted-perfect-square-equation} considerably to a statement comparing characteristic polynomials or eigenvalues.
Given any $\bd$ in $\RR^{n \times n}$  satisfying $\bd^\tr=\varepsilon \bd$ with
$\varepsilon$ in $\RR$, one has $\bd \bd^\tr=\varepsilon \bd^2$.  
Therefore if $\bd$ has nonzero eigenvalues 
$\lambda_1,\ldots,\lambda_r$,
 then $\bd \bd^\tr$ has nonzero eigenvalues 
 $\Lambda_1,\ldots,\Lambda_r$ where
 \begin{align}
  \Lambda_i &= \varepsilon \lambda_i^2,\notag\\
  \lambda_i &=\pm \sqrt{\frac{\Lambda_i}{\varepsilon}}.
\label{square-root-eigenvalues}
 \end{align}
Thus the spectrum of $\bd$ determines 
that of $\bd \bd^\tr$ uniquely, but  the spectrum of $\bd \bd^\tr$ does not in general
determine the $\pm$ signs in \eqref{square-root-eigenvalues}
without further information.
One such situation with further information appears in
Example~\ref{simplex-middle-boundary-charpoly-example} below.
Another occurs when $\bd^\tr=-\bd$ so that $\varepsilon=-1$, 
where the {\it spectral theorem} implies that $r$ is even, and that the
$\lambda_i$ are purely imaginary and occur in complex conjugate
pairs $\pm i \sqrt{\Lambda_i}$.
\end{remark}

\section{Topological motivation: a review of higher-dimensional trees}
\label{trees-section}

Our motivation is the enumeration of 
higher-dimensional spanning trees in cell complexes, as in the
groundbreaking papers of Kalai \cite{Kalai} and Bolker \cite{Bolker}, and with many further developments since; see, e.g. \cite{Adin, BajoBurdickChmutov, CCK, DKM-Simplicial,  DKM-Cellular,  DKM-Cutflow, KrushkalRenardy, Lyons, Maxwell}.   We review here 
trees in higher dimension, and explain 
a (known) further factorization for
the formula  \eqref{reformulated-weighted-pdet-expansion} in the topological setting,
even without any duality hypotheses.

We start by setting up notation for CW-complexes; see, e.g.,  \cite{FritschPiccinini,Hatcher, LundellWeingram, Munkres}. 
A finite CW-complex $S$ has {\it (augmented, integral) cellular chain complex}
\begin{equation}
\label{cellular-chain-complex}
\cdots \longrightarrow 
C_i(S,\ZZ) \xrightarrow{\bd_i}
C_{i-1}(S,\ZZ) \xrightarrow{\bd_{i-1}}
\cdots  \xrightarrow{\bd_1}
C_0(S,\ZZ) \xrightarrow{\bd_0}
C_{-1}(S,\ZZ)=\Zz \to 0.
\end{equation}
We will always be working with integer coefficients, so
we use the abbreviated notation $\HH_i(S)$ 
for the reduced homology $\HH_i(S,\ZZ):=\ker(\bd_i)/\im(\bd_{i+1})$,
with $i \geq -1$.

Let $f_j(S)$ denote the number of $j$-dimensional cells of $S$, and let $S^{(i)}$ denote the {\it $i$-skeleton} of $S$, that is, the subcomplex
of $S$ consisting of all cells having dimension at most $i$.

\begin{definition}
\label{tree-definition}
The CW-complex $S$ is \emph{$(i-1)$-acyclic} if $\HH_j(S)=0$ for $-1 \leq j \leq i-1$.  
For $S$ an $(i-1)$-acyclic CW-complex, a
subcomplex $T \subseteq S$ is an \emph{$i$-dimensional (spanning) tree}, or simply an \emph{$i$-tree}, if $S^{(i-1)}\subseteq T\subseteq S^{(i)}$
and $T$ satisfies the following three conditions
(of which any two imply the third):
\begin{itemize}
\item[(i)] $\HH_i(T)=0$;
\item[(ii)] $\HH_{i-1}(T)$ is finite;
\item[(iii)] the number of $i$-cells in $T$ equals the rank of $\bd_i$, namely
\begin{equation}
\label{rank-of-i-th-boundary}
\rank(\bd_i):=\sum_{j=-1}^{i-1} (-1)^{i-1-j} f_{j}(S).
\end{equation}
\end{itemize}
\end{definition}

These three conditions can alternatively be phrased as saying that
$T$ is $\QQ$-acyclic, that is, its (reduced) homology groups with $\QQ$-coefficients
all vanish.  Another equivalent phrasing is that the subcomplex $T$ 
is an $i$-tree for $S$ if and only if the
$i$-cells in $T$ index a subset of the columns of $\bd_i$ which give
a column-basis for~$\bd_i$ in the sense of 
Definition~\ref{row-column-basis-definition}.  When $S$ is a connected graph (i.e., a 1-dimensional complex that is 0-acyclic), the definition reduces to the usual graph-theoretic definition of a spanning tree.

As a consequence of Definition~\ref{tree-definition}, most groups $\HH_j(T)$ for an $i$-tree $T$ vanish:  
\begin{itemize}
\item $\HH_j(T)=0$ for $j > i$, because $T$ is $i$-dimensional, and 
\item $\HH_j(T)=\HH_j(S)=0$ for $j \leq i-2$, as $T$ and $S$ have the same $(i-1)$-skeleton.
\end{itemize}
The only potentially nonvanishing homology group for $T$ is the finite group $\HH_{i-1}(T)$.

\begin{definition}
For $i\geq 0$ and $S$ an $(i-1)$-acyclic CW-complex, the $i^{th}$ \emph{torsion tree enumerator} of~$S$ is
$$
\tau_i(S) :=\sum_{i\text{-trees }T\subseteq S} |\HH_{i-1}(T)|^2.
$$
More generally, letting $\xx=(x_1,\ldots,x_m)$ be a set of
variables indexing the $i$-cells of $S$, the \emph{$i^{th}$ weighted torsion tree enumerator} is
$$
\tau_i(S,\xx)  :=\sum_{i\text{-trees }T\subseteq S} \xx^T |\HH_{i-1}(T)|^2  
$$
where $\xx^T:=\prod_{j \in T} x_j$ is the product of variables corresponding to the $i$-cells contained in $T$.
\end{definition}

This enumerator for spanning trees arises in higher-dimensional generalizations of the Matrix-Tree Theorem, as we will explain.
Bajo, Burdick and Chmutov note \cite[Theorem 3.2]{BajoBurdickChmutov} that
$\tau_i(S)$ is the $x=y=0$ evaluation of a polynomial in $x,y$ that they call the {\it modified Tutte-Krushkal-Renardy} polynomial of the skeleton $S^{(i)}$, closely related to a polynomial 
introduced by Krushkal and Renardy in \cite{KrushkalRenardy}.

\begin{theorem}
\label{Adin-tree-pdet-factorization}
Let $i\geq0$ and let $S$ be an $(i-1)$-acyclic CW-complex with $f_{i-1}(S)=n$ and $f_i(S)=m$,
so that $\bd:=\bd_i\in\ZZ^{n \times m}$.
Let $\xx=(x_1,\ldots,x_n)$ and $\yy=(y_1,\ldots,y_m)$ be variables 
indexing the $(i-1)$-cells and $i$-cells of $S$, respectively, and let $X=\diag(\xx)$ and $Y=\diag(\yy)$.
Then
\begin{equation} \label{cellular-MTT}
\pdet(\bd \bd^\tr)= \tau_{i-1}(S) \cdot \tau_i(S).
\end{equation}
Moreover, if we let $L=X^{\half} \bd Y \bd^\tr X^{\half}$ be the weighted combinatorial Laplacian, then
\begin{equation} \label{doubly-weighted-tree-enum}
\pdet(L)= \xx^{[n]} \cdot \tau_{i-1}(S,\xx^{-1}) \cdot \tau_i(S,\yy).
\end{equation}
\end{theorem}

Theorem~\ref{Adin-tree-pdet-factorization} originates in the work of Adin~\cite[Thm.~3.4]{Adin} on a special class of simplicial complexes.  It was subsequently generalized by various authors, e.g., \cite{CCK, DKM-Simplicial,  DKM-Cellular,  Lyons, Maxwell, Petersson}.
The unweighted formula~\eqref{cellular-MTT} appears in many of these sources, as do
several formulas for weighted enumeration of $i$-trees.
To our knowledge, no explicit equivalent of the formula \eqref{doubly-weighted-tree-enum}
for \emph{simultaneous} weighted enumeration of $i$- and $(i-1)$-trees
has previously appeared in the literature.  On the other hand, the proof runs along the same lines laid down by Adin and subsequently
explained in detail in many other sources, so we only sketch it here.  The first key observation is that
a set $J$ of columns of $\bd$ is a column basis if and only if the corresponding $i$-faces are the facets of an $i$-tree $T$,
and a set $I$ of rows is a row basis if and only if the corresponding $(i-1)$-faces form the \emph{complement} of the facets of an $(i-1)$-tree $T'$.
(The latter assertion is an instance of Gale duality; see \cite[\S 2.2]{Oxley}, \cite[\S 8.1]{OMBook}.)
Thus equation~\eqref{reformulated-weighted-pdet-expansion} can be rewritten as a sum over pairs $(T,T')$ of $i$- and $(i-1)$-trees.  The second key point is that
\begin{equation}
\label{pair-sequence-consequence}
|\det \bd_{I,J}| ~=~ |\HH_{i-1}(T,T')| ~=~ | \HH_{i-1}(T) | \cdot| \HH_{i-2}(T') |,
\end{equation}
where the second equality can be deduced from the homology long exact sequence for the pair $T'\subset T$.  The same argument goes through
upon replacing $\bd$ with its doubly weighted analogue $X^{\half}\bd Y^{\half}$.

When $S$ is a connected graph ($i=1$), a $0$-spanning tree is simply a vertex.  Thus $\tau_0(S)$ is the number
of vertices, and the unweighted formula is one form of the classical matrix-tree theorem.

\section{Self-dual CW-complexes} 
\label{duality-section}

We next consider stronger assumptions on our CW-complex, some technical, and
some concerning symmetry.

\begin{definition} \label{self-dual-ball}
A {\it self-dual $d$-ball}  is  a pair $(S,\alpha)$ where $S$ is a CW complex, and 
$\alpha$ is a self-map of the face poset $P$ of $S$ defined by the inclusion order on the cells, such that
\begin{itemize}
\item $S$ is a {\it regular} CW-complex (i.e., its attaching maps are homeomorphisms; see \cite{Bjorner,OMBook,FritschPiccinini,LundellWeingram}),
\item $S$ is homeomorphic to a $d$-dimensional ball, and
\item $\alpha$ is anti-automorphism; i.e $\sigma\overset{\alpha}{\longmapsto}\tilde\sigma$ satisfies
$\sigma\subseteq\tau$ if and only if $\tilde\sigma\supseteq\tilde\tau$.
\end{itemize}
\end{definition}

\noindent
Note that since the empty cell $\varnothing$ of
dimension $-1$ is a bottom element in $P$, there will be a unique top
element $\tilde{\varnothing}$ in $P$, which must index the unique $d$-cell of $S$,
that is, the interior of the $d$-ball.   For this reason, a self-dual $d$-ball $(S,\alpha)$ 
is uniquely determined by its boundary $(d-1)$-sphere $\Bd S$ together with the restriction of
$\alpha$ to the {\it proper part} of $P$, indexing non-empty proper cells.

\begin{example}
\label{self-dual-polytopes-example}
{\it Self-dual polytopes}.
A $d$-dimensional convex polytope $\PPP$ gives rise to a regular CW $d$-ball,
namely the cell complex of its faces.
If $\PPP$ is embedded in $\RR^d$
with the origin in its interior, then its {\it polar dual} polytope 
(see Ziegler \cite[Lec.~2]{Ziegler}) is 
$$
\PPP^\diamond:=\{ \yy \in \RR^d: 
             \xx \cdot \yy \leq 1 \text{ for all }\xx \in \PPP \}.
$$
The face poset of $\PPP^\diamond$ is the opposite $P^{\op}$ of the
face poset of $P$; see \cite[Cor. 2.14]{Ziegler}.  

A polytope $\PPP$ is called \emph{self-dual} if there is a poset isomorphism $P \rightarrow P^{\op}$.  In this case, the face complex is a self-dual $d$-ball as in Definition~\ref{self-dual-ball}.  Some families and examples:
\begin{itemize}
\item[(a)] An $m$-sided polygon in $\RR^2$. 
\item[(b)] The pyramid over an $m$-sided polygon in $\RR^3$.
\item[(c)] {\it Elongated} and {\it multiply elongated} pyramids over an $m$-sided polygon in $\RR^3$.
\item[(d)] The {\it diminished trapezohedron}\footnote{The case $m=6$ is depicted via its {\it Schlegel diagram} \cite[\S 5.2]{Ziegler} in Example~\ref{self-dual-plane-graphs-example} below.} over an $m$-sided polygon in $\RR^3$.
\item[(e)] The self-dual {\it regular polyhedron} in $\RR^4$, called the {\it 24-cell}.
\item[(f)] Generalizing (b), any pyramid in $\RR^{d+1}$ over a self-dual polytope in $\RR^d$;
see also \cite[Example 3.2]{Maxwell}.
\item[(g)] Specializing (f), a simplex with $n$ vertices in $\RR^{n-1}$ is an iterated pyramid over a $1$-polytope.
\end{itemize}
Families (a), (b), (c) from the above list are illustrated below, with $m=5$:
$$
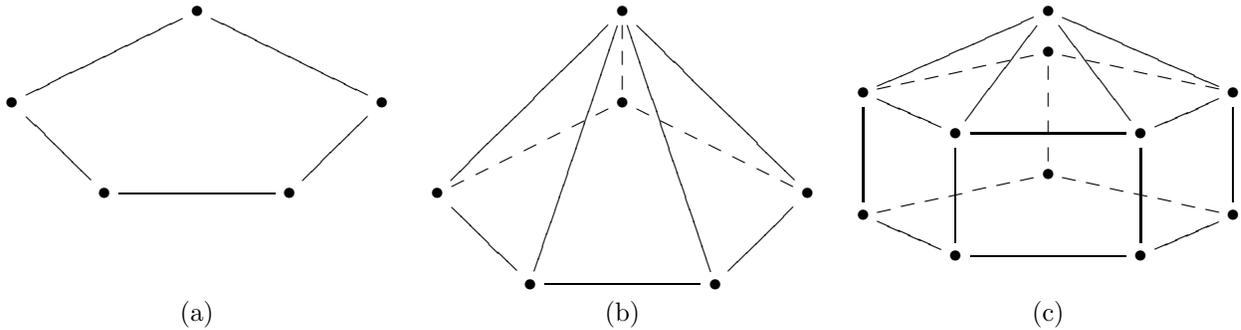
\begin{array}{ccc}
\xymatrix{
  &  &\bullet\ar@{-}[dll]\ar@{-}[drr]&  &\\
\bullet \ar@{-}[dr]& & & &\bullet\ar@{-}[dl] \\
&\bullet\ar@{-}[rr]& &\bullet &\\
}
&
\xymatrix{
  &  &\bullet\ar@{-}[ddll]\ar@{-}[dddl]\ar@{--}[d]\ar@{-}[dddr]\ar@{-}[ddrr]&  &\\
  &  &\bullet\ar@{--}[dll]\ar@{--}[drr]&  &\\
\bullet \ar@{-}[dr]& & & &\bullet\ar@{-}[dl] \\
&\bullet\ar@{-}[rr]& &\bullet &\\
}
&
\xymatrix@R=5pt{
  &  &\bullet\ar@{-}[ddll]\ar@{-}[dddl]\ar@{--}[d]\ar@{-}[dddr]\ar@{-}[ddrr]&  &\\
  &  &\bullet\ar@{--}[dll]\ar@{--}[ddd]\ar@{--}[drr]&  &\\
\bullet \ar@{-}[dr]\ar@{-}[ddd]& & & &\bullet\ar@{-}[dl]\ar@{-}[ddd] \\
&\bullet\ar@{-}[rr]\ar@{-}[ddd]& &\bullet \ar@{-}[ddd]&\\
  &  &\bullet\ar@{--}[dll]\ar@{--}[drr]&  &\\
\bullet \ar@{-}[dr]& & & &\bullet\ar@{-}[dl] \\
&\bullet\ar@{-}[rr]& &\bullet &\\
}\\
\text{(a)} & \text{(b)} & \text{(c)} 
\end{array}
$$
\end{example}

\begin{example}
\label{self-dual-plane-graphs-example}
{\it Self-dual plane graphs.}
A {\it plane graph} is a (finite) graph $G=(V,E)$ with vertices $V$ and
edges $E$, together with a choice of an
embedding in the plane $\RR^2$ that has no edges crossing in their interiors.
Removing the embedded graph from $\RR^2$ results in several connected components, called {\it regions} or {\it faces}.  These faces are the vertex set $V^*$ for 
the {\it plane dual} graph $G^*=(V^*,E^*)$, having an edge $e^*\in E^*$
for every edge $e\in E$, where the two endpoints for $e^*$ correspond to the (possibly identical)
regions on either side of the edge~$e$.  One can always embed $G^*$ in the plane
with a vertex inside each region of $G$, and with the edge~$e^*$ crossing~$e$ transversely.
One can consider both $G, G^*$ as embedded on the $2$-sphere which is the one-point compactification of $\RR^2$, and if there is a self-homeomorphism of this $2$-sphere that sends~$G$ to~$G^*$, we will say
that $G$ is a {\it self-dual plane graph}.

Any plane graph gives rise in this way to a $CW$ $2$-sphere,
but not all of them are {\it regular} $CW$; one must first impose
some vertex-connectivity on $G$.  

\begin{definition}
A graph $G=(V,E)$ is {\it $k$-vertex-connected} if $|V| \geq k+1$ and deleting any subset of vertices
$V' \subset V$ with $0 \leq |V'| \leq k-1$ leaves a connected graph.
\end{definition}

\noindent
In the following proposition, assertion (i) is not hard to check,
and (ii) is a result of Steinitz  \cite[Lecture 4]{Ziegler}.

\begin{proposition}
Consider the $CW$ $2$-sphere $S$ associated to a plane graph $G$.
\begin{itemize}
\item[(i)] $S$ is regular $CW$ if only if $G$ 
has no self-loops,
has at least two edges, and is $2$-vertex-connected.
\item[(ii)]
$S$ is cellularly homemorphic to the boundary of a $3$-dimensional polytope 
if and only if $G$ 
has no self-loops, no parallel edges, and is $3$-vertex-connected.
\end{itemize}
\end{proposition}


\noindent
For example, the $1$-skeleton of the 
$3$-dimensional self-dual polytope 
appearing in Example~\ref{self-dual-polytopes-example}(d) above with $m=6$,
the {\it diminished trapezohedron over a hexagon}, is shown here as
a self-dual plane graph:
$$
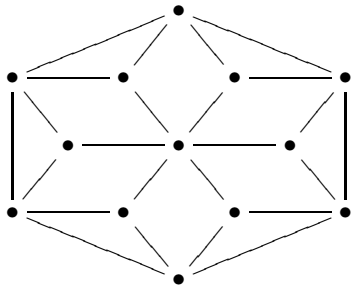
\xymatrix@R=15pt@C=10pt{
 & & &\bullet\ar@{-}[dlll]\ar@{-}[dl]\ar@{-}[dr]\ar@{-}[drrr]& & & \\
\bullet\ar@{-}[dd]\ar@{-}[dr]\ar@{-}[rr]& &\bullet& &\bullet& &\bullet\ar@{-}[dd]\ar@{-}[dl]\ar@{-}[ll]\\
 &\bullet& &\bullet\ar@{-}[rr]\ar@{-}[ll]\ar@{-}[dr]\ar@{-}[dl]\ar@{-}[ur]\ar@{-}[ul]& &\bullet& \\
\bullet\ar@{-}[rr]\ar@{-}[ur]& &\bullet& &\bullet& &\bullet\ar@{-}[ll]\ar@{-}[ul]\\
 & & &\bullet\ar@{-}[ulll]\ar@{-}[ul]\ar@{-}[ur]\ar@{-}[urrr]& & & \\
}
$$

\end{example}

There are two reasons why we have restricted attention to self-dual $d$-balls.  
First, the topology of any regular CW~complex~$X$ is determined combinatorially by its face poset $P$, as we now explain.  Consider the
\emph{order complex} $\Delta(P)$, the 
abstract simplicial complex
whose vertices are the elements of $P$ (i.e., the cells of $X$), and whose simplices are the subsets 
$\{\sigma_1 < \ldots < \sigma_\ell\}$ that are totally ordered
in $P$.  In terms of cells, this means that
$\sigma_i\subseteq\Bd\sigma_j$ for $i < j$
(where $\Bd$ means topological boundary).  Meanwhile, the
regular CW~complex~$X$
has a triangulation, its 
\emph{barycentric subdivision} $\Sd X$, which is
isomorphic as a simplicial complex to $\Delta(P)$;
that is, the geometric realization $|\Delta(P)|$ is homeomorphic to $X$
\cite[Theorem~1.7]{LundellWeingram},
\cite[Theorem~3.4.1]{FritschPiccinini}, \cite{Bjorner}.  The {\it barycenter} $b_\sigma$ 
of the cell $\sigma$ is the point of $|\Delta(P)|$ (or the vertex in $\Sd X$)
that corresponds to the vertex of $\Delta(P)$ indexed by $\sigma$.  Any poset anti-automorphism $\alpha$ of $P$ maps chains to chains, hence induces a simplicial automorphism, which we will denote by $\hat\alpha$, on the face poset of $\Sd X$ ($\isom\Delta(P)$).

The second point about self-dual $d$-balls comes from considering
subcomplexes, corresponding to {\it order ideals} of the poset $P$.
The anti-automorphism $\alpha$ of $P$ 
lets one convert {\it co-complexes}, that is, complements of subcomplexes, or 
{\it order filters} in the face poset $P$, 
back into complexes or order ideals.

\begin{definition} \label{define-Alexander-dual}
Let $(S,\alpha)$ be a self-dual $d$-ball and $T\subseteq S$ a subcomplex.
The {\it Alexander dual} (or \emph{blocker}) of $T$ in $S$ is
$$
T^\vee:= \{ \text{cells }\sigma\text{ of }S: \alpha(\sigma) \not\in T\}.
$$
\end{definition}

\begin{remark}
It follows from the definition that $(T^\vee)^\vee=\alpha^{-2}(T)$.  
The poset anti-automorphism $\alpha$ need not be an involution, so it is
not necessarily the case that $(T^\vee)^\vee=T$.  On the other hand,
$\alpha^2$ is a poset {\it automorphism} of $P$.  Hence $T$ and $(T^\vee)^\vee$ are
regular cell complexes with isomorphic face posets, and are therefore homeomorphic.
\end{remark}


The next result shows that half the torsion tree enumerators $\tau_i(S,\xx)$ 
in a self-dual $d$-ball $S$ determine the rest.  It was observed
by Kalai in \cite[\S 6]{Kalai} for the case that $S$ is a simplex, as in Example~\ref{self-dual-polytopes-example}(g).

\begin{proposition}
\label{Alexander-dual-higher-dimensional-trees}
Let $(S,\alpha)$ be a self-dual $d$-ball, and let $i,j\geq 0$ with $i+j=d-1$.
Then a subcomplex $T \subset S$ is an $i$-tree in $S$ if and only if $T^\vee$ is a $j$-tree.  Furthermore, when these conditions hold, one has
$$
\HH_{i-1}(T) \cong \HH_{j-1}(T^\vee).
$$
In particular, $\tau_i(S) = \tau_j(S)$.  More generally, label the $i$-cells and $j$-cells
with the variables $\xx=(x_1,\ldots,x_n)$ so that each cell
$\sigma$ and its opposite cell $\tilde\sigma$ receive the same label.
Then 
\begin{equation} \label{Alex-dual-weighted}
\tau_i(S,\xx) = \xx^{[n]} \tau_j(S,\xx^{-1})
\end{equation}
where $\tau_j(S,\xx^{-1})$ denotes the rational function obtained from $\tau_j(S,\xx)$ by replacing each $x_i$ with $x_i^{-1}$.
\end{proposition}

\begin{proof}
First note that, for $i+j=d-1$,
a subcomplex $T \subseteq S$ contains $S^{(i)}$ if and only if $T^\vee$ is at most $j$-dimensional, and swapping roles, 
$T^\vee$ contains $S^{(j)}$ if and only if $T$ is at most $i$-dimensional.

We claim\footnote{This claim is known \cite[Lemma~6.2]{ALRS}, \cite[Lemma~4.27]{OMBook}, \cite[Lemma~7]{KrushkalRenardy};
we include the proof for the sake of completeness.} that $T^\vee$ is homeomorphic to a deformation retract 
of $(\Bd S) \setminus T$.
To justify this claim, it suffices to replace $T^\vee$ and $T$ with
their subdivisions $\Sd T^\vee$ and $\Sd T$ inside $\Sd \Bd S$, and
one can also replace  $\Sd T^\vee$ with the isomorphic subcomplex
$\hat\alpha(\Sd T^\vee)$.
By definition, every simplex $\sigma$ in $\Sd \Bd S$ is a simplicial
join $\sigma=\sigma_1 * \sigma_2$ of two of its opposite faces
$\sigma_1, \sigma_2$, lying inside  $\Sd T$ and $\hat\alpha(\Sd T^\vee)$
respectively. Thus one can perform a straight-line 
deformation retraction in the complement 
$|\sigma| \setminus |\sigma_1|$ onto $\sigma_2$,
and these retractions are all simultaneously coherent for every $\sigma$ in $\Sd \Bd S$, giving a retraction of 
$|\Sd \Bd S| \setminus |\Sd T|$ onto $\hat\alpha(\Sd T^\vee)$.

With this claim in hand, since $\Bd S$ is a $(d-1)$-sphere, Alexander duality
\cite[\S 71]{Munkres}, using reduced homology with $\QQ$ coefficients, implies
that the homology over $\QQ$ for $T$ vanishes entirely 
if and only if the same holds for $T^\vee$.  
Hence $T$ is an $i$-tree if and only if $T^\vee$ is a $j$-tree.

For the last assertion, assume $T, T^\vee$ are
$i$-trees and $j$-trees with $i+j=d-1$.  Alexander duality for reduced homology and 
cohomology with $\ZZ$ coefficients implies that
$\HH_{j-1}(T^\vee) \cong H^{i}(T)$, while the universal coefficient theorem for cohomology \cite[\S 53]{Munkres} describes $H^{i}(T)$ via the
split short exact sequence
$$
0 \rightarrow \Ext^1(\HH_{i-1}(T),\ZZ)
\rightarrow H^i(T)
\rightarrow \Hom( \HH_i(T),\ZZ)
\rightarrow 0.
$$
Here $\Hom( \HH_i(T),\ZZ)$ vanishes since $\HH_i(T)=0$, 
and $\Ext^1(\HH_{i-1}(T),\ZZ)  \cong \HH_{i-1}(T)$ 
since $\HH_{i-1}(T)$ is a finite abelian group.
Thus $\HH_{j-1}(T^\vee) \cong \HH_{i-1}(T)$, as desired.
\end{proof}

\begin{remark}
Equation~\ref{Alex-dual-weighted} is also closely related to the Duality Theorem for (modified) Tutte-Krushkal-Renardy
polynomials proven by Bajo, Burdick and Chmutov \cite[Theorem 3.4]{BajoBurdickChmutov}.
\end{remark}

\begin{corollary}[\bf Perfect square phenomenon for even-dimensional self-dual CW-balls]
\label{perfect-square-corollary}
Let $(S,\alpha)$ be a self-dual $d$-ball with $d=2k$ even, and let $\bd=\bd_k$.
Then
\begin{equation}
\label{unweighted-Laplace-pdet-as-square}
\pdet(\bd \bd^\tr)=\tau_{k-1}(S) \tau_k(S) = \tau_k(S)^2.
\end{equation}
More generally, let $X=\diag(\xx)$, $Y=\diag(\yy)$, and $L=X^\half \bd Y \bd^{\tr} X^\half$.  Then
\begin{equation}
\label{weighted-Laplace-pdet-as-square}
\pdet(L)=\xx^{[n]} \tau_{k-1}(S,\xx^{-1}) \tau_k(S,\yy) = \tau_k(S,\xx) \tau_k(S,\yy).
\end{equation}
\end{corollary}
\begin{proof}
Combine Theorem~\ref{Adin-tree-pdet-factorization} with equation~\eqref{Alex-dual-weighted}.
\end{proof}

\noindent
We would like to express the square root of 
$\pdet(\bd \bd^\tr)$ as the pseudodeterminant of~$\bd$.  This requires
an extra geometric hypothesis, {\it antipodal} self-duality,
which is the subject of the next section.

\begin{example}
\label{polygon-example}
Let $n\geq3$ and let $S$ be the 2-dimensional cell complex whose geometric realization is an $n$-sided polygon.
Recall from Example~\ref{self-dual-polytopes-example}(a) that $S$ is a self-dual 2-ball.
We have $\tau_0(S)=n$ (because an 0-tree is a vertex) and $\tau_1(S)=n$ (because every set of $n-1$ edges forms a 1-tree).  Indeed, $\pdet(\bd\bd^\tr)=n^2$, and if we weight the vertices and edges by indeterminates $x_1,\dots,x_n$ and $y_1,\dots,y_n$, then
\[\pdet(X^{\half}\bd Y\bd^\tr X^{\half}) ~=~
\left(\sum_{i=1}^n x_i\right)
\left(\sum_{i=1}^n y_1\cdots\widehat{y_i}\cdots y_n\right).\]
These factors are $\tau_0(S,\xx)$ and $\tau(S,\yy)$ respectively.  In this case, \emph{any} permutation $\pi$ of $[n]$ yields
\[\tau_0(S,\xx) = x_1\cdots x_n \left[ \tau_1(S,\yy) \right]_{y_i=x_{\pi(i)}^{-1}}\]
as in Corollary~\ref{perfect-square-corollary} (even if the pairing between vertices and edges given by $\pi$ is not an anti-automorphism
of the face poset of~$S$).
\end{example}

\begin{remark}
Let $(S,\alpha)$ be a self-dual $d$-ball with $d=2k$ even.
Given a $k$-tree $T$ in~$S$, 
Proposition~\ref{Alexander-dual-higher-dimensional-trees} says that
the Alexander dual $T^\vee$ is a $(k-1)$-tree
having  $\HH_{k-1}(T) \cong \HH_{k-2}(T^\vee)$, so that
$$
|\HH_{k-1}(T)|^2=
|\HH_{k-1}(T)|\cdot |\HH_{k-2}(T^\vee)|
=|\HH_{k-1}(T,T^\vee)|
$$
using \eqref{pair-sequence-consequence} for the second equality.
Thus one can reinterpret the ``square roots'' $\tau_k(S)$ and
$\tau_k(S,\yy)$ in Corollary~\ref{perfect-square-corollary} as follows:
$$
\begin{aligned}
\tau_k(S) &= \sum_{\substack{k\text{-trees }\\ T\text{ in }S}} |\HH_{k-1}(T,T^\vee)|\\
\tau_k(S,\yy) &= \sum_{\substack{k\text{-trees } \\T\text{ in }S}}  \yy^T |\HH_{k-1}(T,T^\vee)|.
\end{aligned}
$$
These simplicial pairs $(T,T^\vee)$ have a similar ``self-dual'' flavor to
the objects that are counted by the square roots of
spanning tree counts in \cite{Maxwell}.

For example, when $S$ is
the simplex $\Delta_{2k+1}$ on $2k+1$ vertices,
one can reinterpret Kalai's generalization of Cayley's
formula~\cite[Thm.~3]{Kalai}.  In our notation, Kalai's theorem is
$$
\tau_k(\Delta_{2k+1})
:= \sum_{\substack{k\text{-trees } \\T\text{ in }\Delta_{2k+1}}}   
  |\HH_{k-1}(T)|^2 \\
=\sum_{\substack{k\text{-trees } \\T\text{ in }\Delta_{2k+1}}}    
  |\HH_{k-1}(T,T^\vee)| 
= (2k+1)^{\binom{2k-1}{k-1}}.
$$
In fact, he generalized a version of the Cayley-Pr\"ufer formula
for trees.  His result factors the following specialization of 
$\tau_k(\Delta_{2k+1},\yy)$, that counts
$k$-trees according to their vertex degrees:
$$
\begin{aligned}
\left[ \tau_k(\Delta_{2k+1},\yy) \right]_{\yy \rightarrow \zz}
&:=\sum_{\substack{k\text{-trees } \\T\text{ in }\Delta_{2k+1}}}  
   \zz^{\deg(T)} |\HH_{k-1}(T)|^2 \\
&=\sum_{\substack{k\text{-trees } \\T\text{ in }\Delta_{2k+1}}}  
  \zz^{\deg(T)} |\HH_{k-1}(T,T^\vee)| 
= \left( z_1 z_2 \cdots z_{2k+1}
          (z_1+\cdots+z_{2k+1})\right)^{\binom{2k-1}{k-1}} .
\end{aligned}
$$
Here $\zz^{\deg(T)}:=\prod_{j=1}^n z_j^{d_j}$, where
$d_j$ is the number of $k$-dimensional simplices in $T$
containing vertex $j$, and the specialization map
$\left[-
\right]_{\yy \rightarrow \zz}$ sends each $y_i$ to $z_{j_1} \cdots z_{j_{k+1}}$,
where $y_i$ indexes the $k$-simplex $\sigma$ with vertex set $\{j_1,\ldots,j_{k+1}\}$.
\end{remark}

\section{Antipodal self-duality} 
\label{antipodally-self-dual-section}

We now consider self-dual $d$-balls that are {\it antipodally} self-dual
in the sense of \cite[Definition 3.1]{Maxwell}.
The main result of this section, Theorem~\ref{antipodally-self-dual-square-root-theorem}, asserts that for an antipodally self-dual complex, the number $\pdet(\bd)$, and more generally the polynomials $\pdet(Y\bd^\tr)$ and $\pdet(X\bd)$, can be interpreted directly as spanning tree enumerators.
(In contrast, for a ball that is self-dual but not antipodally self-dual, these polynomials do not have an evident combinatorial meaning.)

In order to define antipodal self-duality,
we first recall the topological notion of a \emph{dual block decomposition}.
Let $X$ be a regular CW-complex with face poset $P$.
Consider the \emph{chains} in $P$, i.e., its totally ordered subsets
$\sigma_1 < \ldots < \sigma_\ell$.
Each such chain corresponds to a simplex in the barycentric subdivision $\Sd X$.
The \emph{(open) dual block} $\opendualblock(\sigma)$ of a cell $\sigma\in X$
is the union of the interiors of all simplices arising from chains with $\sigma_1=\sigma$.
In particular, $X$ is the disjoint union of its open dual blocks.
The closure of $\opendualblock(\sigma)$ is the \emph{closed dual block}, denoted by $\dualblock(\sigma)$.
Dual blocks in general CW-complexes can behave badly, but when $X$ is a $k$-manifold without boundary, they are at least (integer) homology $k$-balls; see \cite[\S 64]{Munkres}.  Here we are interested in complexes satisfying the following much stronger condition.

\begin{definition}
A self-dual $d$-ball $(S,\alpha)$ 
is {\it antipodally self-dual} if its
boundary $(d-1)$-sphere $\Bd S$ satisfies the following conditions:
\begin{itemize}
\item the dual block decomposition
$D(\Bd S)$ is also a regular CW-complex,
\item the antipodal map 
$a:|\Bd S|\to|D(\Bd S)|$ 
is a regular cellular isomorphism 
$\Bd S \longrightarrow D(\Bd S)$, and
\item 
the antipodal map respects $\alpha$, in the sense that
$a(\sigma) = D(\alpha(\sigma))$ for every cell $\sigma$.
\end{itemize}
\end{definition}
\noindent
Equivalently, $S$ is antipodally self-dual
if it can be embedded as the unit ball in $\RR^d$, with
usual antipodal map $a(x)= -x$,
such that there is a homeomorphism $\varphi:B\to\Delta(P)$
satisfying $\varphi\circ a=|\alpha|\circ\varphi$,
where $|\alpha|:\Delta(P)\to\Delta(P)$ is the simplicial automorphism induced by
the poset anti-automorphism~$\alpha$.

\begin{example}
\label{antipodally-self-dual-polytopes-example}
Not all  of the self-dual polytopes from Example~\ref{self-dual-polytopes-example} are 
{\it antipodally} self-dual:
\begin{itemize}
\item[(a,b,c)] An $m$-sided polygon,
the pyramid over it, and multiply elongated pyramids over it
will all be antipodally self-dual only for $m$ {\it odd}.
\item[(d)] The {\it diminished trapezohedron} over an $m$-sided polygon 
in $\RR^3$ is antipodally self-dual only for $m$ {\it even}.
\item[(e)] The {\it 24-cell} in $\RR^4$ is {\it not} antipodally 
self-dual; see the discussion of self-dual regular polytopes below.
\item[(f)] The pyramid in $\RR^{d+1}$ over an antipodally 
self-dual polytope in $\RR^d$ will be antipodally self-dual.
More generally, taking the pyramid over an antipodally self-dual 
$d$-ball yields an antipodally self-dual $(d+1)$-ball;
see \cite[Example 3.2]{Maxwell}.
\item[(g)] Specializing (f), a simplex is an iterated pyramid,
hence antipodally self-dual.
\item[(h)] A \emph{central reflex} is an
antipodally self-dual graph embedded on a $2$-sphere; these
were studied by Tutte \cite{Tutte}.
\end{itemize}

For $d \geq 2$, one can show that a 
self-dual regular $d$-polytope is antipodally self-dual if and only 
if the antipodal map $-1$ on $\RR^d$ 
is {\it not} an element of its symmetry group, which is always a finite reflection
group, that is, a finite {\it Coxeter group}.  It is known that
$-1$ lies in such a reflection group if and only if its {\it fundamental degrees} \cite[\S 3.7]{Humphreys} 
are all even.  For example:

\begin{itemize}
\item The $24$-cell (type $F_4$) has degrees $(2,6,8,12)$ all even, so is 
not antipodally self-dual.
\item A regular $m$-gon (type $I_2(m)$) has degrees $(2,m)$,
so is antipodally self-dual for $m$ odd.  
\item A regular $(n-1)$-simplex (type $A_{n-1}$) for $n \geq 3$
has degrees $(2,3,\ldots,n)$, so is antipodally self-dual.
\end{itemize}
\end{example}

The second author \cite{Maxwell} showed that certain 
antipodally self-dual $d$-balls with $d=2k+1$ odd carry an orientation
of their cells that make their middle two boundary maps $\bd_k, \bd_{k+1}$ take a certain
highly symmetric form, and from this deduced the following perfect square expression
\cite[Theorem 1.6]{Maxwell} for $\tau_k(S)$:
$$
\tau_k(S) := \sum_{\substack{k\text{-trees }T\\ \text{ in }S}} |\HH_{k-1}(T)|^2 
=\left( 
\sum_{\substack{ \text{self-dual }\\ k\text{-trees }T=T^\vee\\ \text{ in }S}} 
|\HH_{k-1}(T)|
\right)^2.
$$

\begin{lemma}
\label{antipodally-self-dual-middle-boundary}
Let $(S,\alpha)$ be an antipodally self-dual $2k$-ball.  Given
any choice of orientation for its $(k-1)$-cells, one can orient the
$k$-cells so that $\bd_k^\tr=(-1)^k \bd_k$.
\end{lemma}

The proof of Lemma~\ref{antipodally-self-dual-middle-boundary} is  technical
and is deferred to the Appendix.  With such an orientation in hand, we can state the
combinatorial consequences, which follow directly from Corollary~\ref{perfect-square-corollary}.

\begin{theorem}[\bf Perfect square phenomenon for antipodally self-dual CW-balls]
\label{antipodally-self-dual-square-root-theorem}
Let $S$ be an antipodally self-dual $d$-ball with $d=2k$ even,
and orient its cells so that $\bd:=\bd_k$ has $\bd^\tr=(-1)^k \bd$.
Then 
$$
\pdet(\bd) = \tau_k(S) = \tau_{k-1}(S),
$$
so that \eqref{unweighted-Laplace-pdet-as-square} becomes
$$\pdet(\bd\bd^\tr)=(\pdet\bd)^2.$$

More generally, using the anti-automorphism $\alpha$ as a bijection between the
$k$-cells and $(k-1)$-cells with the same set of variables $\xx=(x_1,\ldots,x_n)$,
and defining $L=X^{\half} \bd Y \bd^\tr X^{\half}$,
then the two factors in \eqref{weighted-Laplace-pdet-as-square} are
$$
\begin{aligned}
\pdet(Y \bd^\tr) &= \tau_k(S,\yy) \\
\pdet(X \bd) &= \xx^{[n]} \tau_{k-1}(S,\xx^{-1}) =\tau_k(S,\xx). 
\end{aligned}
$$
\end{theorem}

\begin{question}
Does one similarly obtain a perfect square with an interestingly interpreted 
square root, 
when enumerating spanning
trees of middle dimension in antipodally self-dual $d$-balls for $d \equiv 1 \bmod{4}$?
\end{question}

\noindent
The second author \cite{Maxwell} answered this question for $d\equiv 3 \bmod{4}$,
and Theorem~\ref{antipodally-self-dual-square-root-theorem} covers the cases 
$d \equiv 0,2 \bmod{4}$.

\begin{example}
Let $S$ be a polygon with $n$ sides.  Recall
from Example~\ref{polygon-example} the calculation of the pseudodeterminants of the unweighted and weighted Laplacians of $S$.
In the case that $n=2m+1$ is odd, so that $S$ is antipodally self-dual, we can interpret the pseudodeterminant of the boundary matrix $\bd=\bd_1(S)$
as well.  Label the vertices $v_0,\dots,v_{2m}$
in cyclic order and let $e_{i,i+1}$ be the edge joining $v_i$ and $v_{i+1}$ (with all indices taken modulo~$n$).
Then the map $\sigma\overset{\alpha}{\longmapsto}\tilde\sigma$ given by
\[\tilde{v}_i = e_{i+m,i+m+1}, \quad \tilde{e}_{i,i+1}=v_{i-m}\]
makes $S$ into an antipodally self-dual complex.  
Orienting the edges so that $\bd_1(e_{i,i+1})=v_{i+1}-v_i$ makes the
boundary map $\bd$ antisymmetric,
i.e., $\bd^\tr = (-1)^k \bd=-\bd$.  One can then check that $\rank\bd=n-1$ and
$$
\pdet(\bd)=n-1=\tau_0(S)=\tau_1(S) = \sqrt{\pdet(\bd \bd^\tr)}
$$
and more generally, $L=X^{\half} \bd Y \bd^\tr X^{\half}$ has
$\pdet( L ) = \pdet( X \bd ) \pdet (Y \bd^\tr )$  
with
$$
\begin{aligned}
\pdet (Y \bd^\tr ) &=\tau_1(S,\yy),\\
\pdet( X \bd ) &=\tau_0(S,\xx).
\end{aligned}
$$
\end{example}

\begin{example} \label{duality-for-odd-simplices}
As in Example~\ref{antipodally-self-dual-polytopes-example}(g), let $k\geq1$ and let $S$ be the simplex on $N=2k+1$ vertices, which is an antipodally self-dual $d$-ball with $d=2k=N-1$.  The duality map is simply complementation: $\tilde\sigma=[N]\sm\sigma$ (regarding each face of $S$ abstractly as a subset of $[N]$).
Also, the (skew-)symmetric orientation of Lemma~\ref{antipodally-self-dual-middle-boundary} can be made explicit.
In the standard orientation of $S$, the boundary operators are as follows: if $\sigma=\{v_1,\dots,v_r\}$ with $v_1<\cdots<v_r$, then
\begin{equation}
\label{simplex-boundary-map}
\bd(\sigma) = \sum_{j=1}^r (-1)^{j-1} (\sigma\sm\{v_j\}).
\end{equation}
Write $\bd_k(S)$ as a $\binom{N}{k}\x\binom{N}{k}$ matrix with columns and rows corresponding to $k$-faces and $(k-1)$-faces respectively, so that the $i^{th}$ row and $i^{th}$ column are indexed by complementary faces.  Then, for each $k$-face  $\sigma$, multiply the corresponding column by $(-1)^{\|\sigma\|}$,
where $\|\sigma\|=\sum_{v\in\sigma}v$.  The resulting matrix is then symmetric or skew-symmetric according as $k$ is even or odd; we omit the proof.

\begin{remark}
\label{simplex-middle-boundary-charpoly-example}
In fact, we can completely determine the spectrum of $\bd_k$ in
the situation of Example~ \ref{duality-for-odd-simplices}.  Let $S$ be the 
simplex on $n:=2k+1$ vertices and let
$$
A:=\binom{n}{k+1} =\binom{2k+1}{k+1},
\qquad
B:=\binom{n-1}{k+1} =\binom{2k}{k+1},
\qquad
C:=\binom{n-1}{k} =\binom{2k}{k},
$$
so that  $A=B+C$.  Then:

\begin{proposition} 
 The middle boundary map  
$\bd=\bd_k$ in $\ZZ^{A \times A}$ of $S$, oriented
so that $\bd^\tr = (-1)^k\bd$, satisfies
$$
\begin{aligned}
\det(t \one-\bd \bd^\tr)
&=t^B (t + n)^C \\
\det(t \one-\bd)
&= \begin{cases}
t^B \left( t \pm i \sqrt{n} \right)^{\frac{C}{2}} 
=t^B ( t^2 + n)^{\frac{C}{2}}
& \text{ if } k \text{ is }odd,\\
t^B ( t \pm \sqrt{n})^{\frac{C}{2}}
=t^B ( t^2 - n)^{\frac{C}{2}}
& \text{ if } k \text{ is }even.\\
\end{cases}
\end{aligned}
$$
\end{proposition}
\begin{proof}
The first equality is equivalent to the assertion that $\bd \bd^\tr$ has only one nonzero eigenvalue $n$, with multiplicity $C$.  This follows immediately from
\cite[Theorem 1.1]{DuvalR}.

The spectrum for $\bd$ can be deduced 
from the discussion at the end of Remark~\ref{square-root-of-spectrum-remark}
as follows.  When $k$ is odd, we have $\bd^\tr=-\bd$, so the eigenvalues of $\bd$ come in purely imaginary complex conjugate pairs $\pm i \sqrt{n}$.  
When $k$ is even, we have $\bd^\tr=\bd$, it must be that
$\bd$ has only two nonzero real
eigenvalues $+\sqrt{n}, -\sqrt{n}$, with multiplicities summing to the rank
$C$.  Summing these eigenvalues with multiplicity 
gives the trace of $\bd$, which is $0$, for the following reason: by the definition of $\alpha$,
every diagonal entry of $\bd$ corresponds to a pair of complementary (in particular disjoint) simplices, hence is zero
by~\eqref{simplex-boundary-map}.
Hence both multiplicities are $\frac{C}{2}$.
\end{proof}

\end{remark}
\end{example}

\section{Appendix:  orienting antipodally self-dual balls}
\label{antipodally-self-dual-appendix}

The goal of this appendix is to prove Lemma~\ref{antipodally-self-dual-middle-boundary}, recalled here:

\vskip.1in
\noindent
{\bf Lemma~\ref{antipodally-self-dual-middle-boundary}.}
\emph{Let $(S,\alpha)$ be an antipodally self-dual $2k$-ball.  Given
any choice of orientation for its $(k-1)$-cells, one can orient the
$k$-cells so that $\bd_k^\tr=(-1)^k \bd_k$.}

\vskip.1in
\noindent
As in \cite{Maxwell},
we first recall facts about 
orienting cells in regular CW-complexes and cellular boundary maps,
then give the procedure for orienting $k$-cells as in the lemma, and 
finally check that it works.

\subsection{Orientations in regular CW-complexes}
\label{orientation-discussion-subsection}

Useful references are  
Bj\"orner~\cite{Bjorner},
Lundell and Weingram  \cite[Chap. V]{LundellWeingram},
Munkres \cite[\S 39]{Munkres}.

Orienting a $k$-cell $\sigma$ in a CW-complex means choosing one of
the two generators for $\HH_k(\sigma,\Bd \sigma) \cong \ZZ$.   When the
CW-complex comes from an abstract simplicial complex,
and $\sigma$ is an  $k$-simplex with vertices $\{v_0,v_1,\ldots,v_k\}$ 
this choice of generator is equivalent to the choice of a function 
$$
\{\text{linear orderings of }\{v_0,v_1,\ldots,v_k\} \}
\xrightarrow{\sgn}
\{+1,-1\}
$$
that is {\it alternating}, i.e.
$$
\sgn(v_{\sigma(0)},\ldots,v_{\sigma(k)}) 
 = (-1)^\sigma \sgn(v_0,\ldots,v_k),
$$
where $(-1)^\sigma$ denotes the usual sign of the permutation $\sigma$.

In a {\it regular} CW-complex, this choice of
orientation on a $k$-cell $\sigma$ is equivalent to a choice
of orientation for any $k$-simplex $\widehat{\sigma}$ of  its barycentric subdivision
$\Sd \sigma$: 
if $x$ is any point in the interior of $\widehat{\sigma}$, one has 
isomorphisms
$$
\xymatrix{
\HH_k(\sigma,\Bd \sigma) \ar^{\sim}[r]&  \HH_k(\sigma,\sigma\setminus \{x\}) \\
\HH_k(\widehat{\sigma},\Bd \widehat{\sigma}) \ar^{\sim}[r] &  \HH_k(\widehat{\sigma},\widehat{\sigma} \setminus \{x\}) \ar^{\sim}[u]
}
$$
coming from excision of
\begin{itemize}
\item
$\sigma \setminus \widehat{\sigma}\subset \sigma \setminus \{x\}$ for the vertical map,
\item
$\sigma \setminus \Bd \sigma \subset \sigma \setminus \{x\}$ for the top horizontal map, and
\item
$\widehat{\sigma} \setminus \Bd \sigma \subset \widehat{\sigma} \setminus \{x\}$ for the bottom horizontal map.
\end{itemize}

\noindent
Conversely, one has the following.
\begin{proposition}
\label{local-orientation-criterion}
Given a $k$-cell $\sigma$ in a regular CW-complex, 
a collection of orientations on all $k$-simplices 
$\widehat{\sigma}=\{b_0,\ldots,b_{k-1},b_\sigma\}$ inside $\Sd \sigma$ 
comes from one global orientation of $\sigma$ if and only if these orientations
satisfy all local compatibilities
$$
\sgn(b_0,\ldots,b_{i-1},b_i,b_{i+1},\ldots,b_{k-1},b_\sigma)\\
~=~-\sgn(b_0,\ldots,b_{i-1},b'_i,b_{i+1}\ldots,b_{k-1},b_\sigma)
$$
for adjacent $k$-simplices differing in
barycenters $b_i, b'_i$ of $i$-cells with $0 \leq i \leq k-1$.
\end{proposition}


We now describe the cellular boundary map $\partial_k$ explicitly.
This map is the connecting homomorphism
$$
\HH_k(X^{(k)},X^{(k-1)}) \rightarrow \HH_{k-1}(X^{(k-1)},X^{(k-2)})
$$
from the long exact sequence of the triple $(X^{(k)},X^{(k-1)},X^{(k-2)})$.
(Recall that $X^{(k)}$ denotes the $k$-skeleton of a CW-complex $X$.)
When $X$ is a {\it regular} CW-complex, one has for each $k$ an isomorphism
$$
\HH_k(X^{(k)},X^{(k-1)})
\cong 
\bigoplus_{k\text{-cells }\tau} \HH_k(\tau,\Bd \tau),
$$
and $\bd_k$ takes the form
$$
\displaystyle \bigoplus_{k\text{-cells } \tau} \HH_k(\tau,\Bd \tau) 
\xrightarrow{\bd_k}
\displaystyle \bigoplus_{(k-1)\text{-cells } \sigma} \HH_{k-1}(\sigma,\Bd \sigma)
$$
in which $(\bd_k)_{\sigma,\tau}$ for an oriented $k$-cell $\tau$ and 
$(k-1)$-cell $\sigma$ is $\pm 1$ when $\tau$ contains $\sigma$, 
and $0$ otherwise.  A way to compute 
$(\bd_k)_{\sigma,\tau}$ is as follows.  Choose any 
$(k-1)$-simplex 
\begin{equation}
\label{choice-of-subsimplex}
\widehat{\sigma}=\{b_0,b_1,\ldots,b_{k-1}\} \text{ in }\Sd \sigma,
\end{equation}
where $b_i$ is the barycenter of an $i$-cell inside
$\sigma$ for $i=0,1,\ldots,k-1$, so $b_{k-1}=b_\sigma$ is the
barycenter of $\sigma$.  Then 
\begin{equation}
\label{epsilon-via-subdivision}
(\bd_k)_{\sigma,\tau} = 
\frac{\sgn(b_0,b_1,\ldots,b_{k-1},b_\tau)}
{\sgn(b_0,b_1,\ldots,b_{k-1})}.
\end{equation}

\subsection{Orienting $k$-cells as in Lemma~\ref{antipodally-self-dual-middle-boundary} }
\label{the-orientation-procedure-section}

Let $(S,\alpha)$ be a self-dual $2k$-ball, with a given
orientation for all of its $(k-1)$-cells $\sigma$.  
Once and for all, for each $\sigma$, choose 
a $(k-1)$-simplex $\widehat{\sigma}$ inside $\Sd \sigma$
as in \eqref{choice-of-subsimplex}.
The discussion in \S\ref{orientation-discussion-subsection} above
implies that each such simplex $\widehat{\sigma}$ has also been oriented.

We wish to now orient all the $k$-cells of $S$.
First fix an orientation class $z$ generating
$
\HH_{2k-1}(\Sd \Bd S) \cong \ZZ,
$
and orient all the $(2k-1)$-simplices 
$\{b_0,b_1,\ldots,b_{2k-1}\}$ of $\Sd(\Bd S)$ in such a way that $z$
is their sum with all coefficients $+1$.
These orientations will satisfy
\begin{equation}
\label{antipode-is-degree-one}
\sgn(a(b_0),a(b_1),\ldots,a(b_{2k-1}) )
~=~\sgn(b_0,b_1,\ldots,b_{2k-1})
\end{equation}
because the antipodal map $a$ on the $(2k-1)$-sphere $\Bd S$
has degree $(-1)^{2k}=+1$, so $a(z)=+z$ in $\HH_{2k-1}(\Sd \Bd S)$;
see, e.g., \cite[\S V.4]{LundellWeingram} or
\cite[Theorem 21.3]{Munkres}.

To orient a typical $k$-cell $\tau$ of $S$, we use 
that its dual block $D(\tau)$ has antipodal image
$a(D(\tau))=\widetilde{\tau}$, which is a $(k-1)$-cell, and hence
has already been oriented.   This means that the $(k-1)$-simplex
$\widehat{\widetilde{\tau}}=\{\widetilde{b}_0,\ldots,\widetilde{b}_{k-1}\}$
in $\Sd \widetilde{\tau}$ has also been oriented, and 
any $k$-simplex $\{b_0,\ldots,b_{k-1},b_k\}$ in $\Sd \tau$
satisfies 
$
b_k=b_\tau\left(=a(\widetilde{b}_{k-1})\right)
$
and has join with $\{a(\widetilde{b}_0),\ldots,a(\widetilde{b}_{k-1})\}$
giving a maximal simplex in $\Sd(\Bd S)$:
$$
\{
b_0,\ldots,b_{k-1},
a(\widetilde{b}_{k-1}),\ldots,a(\widetilde{b_0})
\}.
$$
Now orient $\tau$ via orientations of all $k$-simplices $\{b_0,\ldots,b_k\}$ in $\Sd \tau$, by decreeing that
\begin{equation}
\label{decreed-orientation}
\sgn(b_0,\ldots,b_k):=
\frac{
\sgn(
b_0,\ldots,b_{k-1},
a(\widetilde{b}_{k-1}),\ldots,a(\widetilde{b_0})
)}{\sgn(\widetilde{b}_0,\ldots,\widetilde{b}_{k-1})}.
\end{equation}
Note that this collection of orientations satisfies the local criteria in Proposition~\ref{local-orientation-criterion}: replacing $b_i$ by $b'_i$ with $0 \leq i \leq k-1$
on the left side of \eqref{decreed-orientation} has the
effect on the right side of doing the same replacement in
the numerator, which will only multiply the numerator by $-1$,
since the orientation
of maximal simplices of $\Sd(\Bd S)$ in the orientation class $z$
is consistent with an orientation on all maximal cells of $\Bd S$.

\subsection{Proof of Lemma~\ref{antipodally-self-dual-middle-boundary}}

We wish to verify that the procedure described in Subsection~\ref{the-orientation-procedure-section} orients
the $k$-cells $\tau$ in such a way that all nested pairs $\sigma \subset \tau$ of
$(k-1)$ and $k$-cells satisfy
$
(\bd_k)_{\sigma,\tau}
=(-1)^k (\bd_k)_{\widetilde{\tau},\widetilde{\sigma}}.
$

If one names the vertex sets in the two $(k-1)$-simplices
$$
\begin{aligned}
\widehat{\sigma}&=\{b_0,\ldots,b_{k-1}\} 
 \text{ in } \Sd \sigma,\\
\widehat{\widetilde{\tau}}&=\{\widetilde{b}_0,\ldots,\widetilde{b}_{k-1}\}  
 \text{ in } \Sd \widetilde{\tau},
\end{aligned}
$$
then one will have these vertex sets for these two $k$-simplices:
$$
\begin{array}{rcl}
\{b_0,\ldots,b_{k-1},b_k\} 
& \text{ in }  \Sd \tau 
& \text{ with }b_k=b_\tau=a(\widetilde{b}_{k-1}),\\
\{\widetilde{b}_0,\ldots,\widetilde{b}_{k-1},\widetilde{b}_k\} 
& \text{ in }  \Sd \widetilde{\sigma}  
& \text{ with }\widetilde{b}_k=b_{\widetilde{\sigma}}=a(b_{k-1}).
\end{array}
$$
Then \eqref{epsilon-via-subdivision} and \eqref{decreed-orientation} will imply
$$
(\bd_k)_{\sigma,\tau}
=\frac{\sgn(b_0,\ldots,b_{k-1},a(\widetilde{b}_{k-1}),\ldots,a(\widetilde{b_0}))}
         {\sgn(b_0,\ldots,b_{k-1})
           \sgn(\widetilde{b}_0,\ldots,\widetilde{b}_{k-1})} \quad \text{ and } \quad
(\bd_k)_{\widetilde{\tau},\widetilde{\sigma}} 
=\frac{\sgn(\widetilde{b}_0,\ldots,\widetilde{b}_{k-1},a(b_{k-1}),\ldots,a(b_0))}
         {\sgn(\widetilde{b}_0,\ldots,\widetilde{b}_{k-1})\sgn(b_0,\ldots,b_{k-1})}.
$$
Therefore
$$
\begin{aligned}
\frac{ (\bd_k)_{\sigma,\tau} }
        { (\bd_k)_{\widetilde{\tau},\widetilde{\sigma}} }
&=\frac{\sgn(b_0,\ldots,b_{k-1},a(\widetilde{b}_{k-1}),\ldots,a(\widetilde{b_0}))}
         {\sgn(\widetilde{b}_0,\ldots,\widetilde{b}_{k-1},a(b_{k-1}),\ldots,a(b_0))} \\
&=\frac{\sgn(a(b_0),\ldots,a(b_{k-1}),\widetilde{b}_{k-1},\ldots,\widetilde{b_0})}
         {\sgn(\widetilde{b}_0,\ldots,\widetilde{b}_{k-1},a(b_{k-1}),\ldots,a(b_0))} 
         =(-1)^k
\end{aligned}
$$
where the second-to-last equality used
\eqref{antipode-is-degree-one} in the numerator, and the last equality is because
reversing $2k$ letters can be done via $k$ transpositions, with sign $(-1)^k$.

\end{document}